\documentclass[12pt]{amsart}

\usepackage{amsmath,amsthm,amssymb,graphicx,kbordermatrix,setspace}
\usepackage[margin=1in]{geometry}
\usepackage[hidelinks]{hyperref}

\newtheorem{theorem}{Theorem}[section]
\newtheorem{proposition}[theorem]{Proposition}
\newtheorem{corollary}[theorem]{Corollary}

\theoremstyle{definition}
\newtheorem{definition}[theorem]{Definition}
\newtheorem{remark}[theorem]{Remark}
\newtheorem{convention}[theorem]{Convention}

\numberwithin{equation}{section}

\DeclareMathOperator{\id}{id}

\newcommand{\A}{\mathcal{A}}
\newcommand{\B}{\mathcal{B}}
\newcommand{\C}{\mathbb{C}}
\newcommand{\Cc}{\mathcal{C}}
\newcommand{\E}{\mathcal{E}}
\newcommand{\F}{\mathbb{F}}
\newcommand{\gl}{\mathfrak{gl}}

\newcommand{\I}{\mathcal{I}}
\newcommand{\Nc}{\mathcal{N}}
\newcommand{\Pc}{\mathcal{P}}
\newcommand{\U}{\mathcal{U}}
\newcommand{\Z}{\mathbb{Z}}
\newcommand{\Zc}{\mathcal{Z}}

\title{Compatibility in Ozsv{\'a}th--Szab{\'o}'s bordered HFK via higher representations}

\author{William Chang}
\address{Department of Mathematics, University of Southern California, 3620 S. Vermont Ave., KAP 104, 
Los Angeles, CA 90089-2532}
\email{chan087@usc.edu}

\author{Andrew Manion}
\address{Department of Mathematics, North Carolina State University, 2108 SAS Hall, Raleigh, NC 27695}
\email{ajmanion@ncsu.edu}

\begin{document}

\begin{abstract}
    We equip the basic local crossing bimodules in Ozsv{\'a}th--Szab{\'o}'s theory of bordered knot Floer homology with the structure of 1-morphisms of 2-representations, categorifying the $U_q(\gl(1|1)^+)$-intertwining property of the corresponding maps between ordinary representations. Besides yielding a new connection between bordered knot Floer homology and higher representation theory in line with work of Rouquier and the second author, this structure gives an algebraic reformulation of a ``compatibility between summands'' property for Ozsv{\'a}th--Szab{\'o}'s bimodules that is important when building their theory up from local crossings to more global tangles and knots.
\end{abstract}

\maketitle

\tableofcontents

\section{Introduction}
Ozsv{\'a}th--Szab{\'o}'s theory \cite{OSzNew,OSzNewer,OSzHolo,OSzHolo2} of bordered knot Floer homology, or bordered HFK, has proven to be highly efficient for computations (see \cite{HFKCalc} for a fast computer program based on the theory). It works by assigning certain dg algebras to sets of $n$ tangle endpoints (oriented up or down) and certain $A_{\infty}$ bimodules to tangles; one recovers HFK for closed knots by taking appropriate tensor products of these bimodules.

In \cite{ManionDecat}, the second author showed that the dg algebras of bordered HFK categorify representations of the quantum supergroup $U_q(\gl(1|1))$ and that the tangle bimodules categorify intertwining maps between these representations. While \cite{ManionDecat} did not consider a categorified action of the quantum group on the bordered HFK algebras, such an action (for Khovanov's categorification $\U$ \cite{KhovOneHalf} of the positive half $U_q(\gl(1|1)^+) = \frac{\C(q)[E]}{(E^2)}$) was defined in \cite{LaudaManion}, compatibly (via \cite{LP, MMW2}) with a more general family of higher actions defined in \cite{ManionRouquier}.

Since Ozsv{\'a}th--Szab{\'o}'s tangle bimodules categorify intertwining maps between representations, it is natural to ask whether the bimodules themselves intertwine the higher actions of $\U$ on the bordered HFK algebras. Since a higher action of $\U$ on a dg algebra $\A$ amounts to a dg bimodule $\E$ over $\A$ together with some extra data, one (roughly) asks whether tangle bimodules $X$ satisfy $X \otimes_{\A} \E \cong \E \otimes_{\A} X$. A structured way to require such commutativity is to equip $X$ with the data of a 1-morphism between 2-representations of $\U$.

The main result of this paper is that one can naturally equip Ozsv{\'a}th--Szab{\'o}'s local crossing bimodules with this 1-morphism structure.

\begin{theorem}\label{thm:IntroMain}
Ozsv{\'a}th--Szab{\'o}'s local crossing bimodules $\Pc$ and $\Nc$, for a positive and negative crossing between two strands, can be equipped with the structure of 1-morphisms of 2-representations over $\U$, encoding the commutativity of $\Pc$ and $\Nc$ with the 2-action bimodule $\E$.
\end{theorem}

In fact, the algebra over which $\Pc$ and $\Nc$ are defined has two natural 2-actions of $\U$, and we prove Theorem~\ref{thm:IntroMain} for both 2-actions. Below we comment a bit more on the motivation and potential applications for Theorem~\ref{thm:IntroMain}, as well as future directions for study.

\begin{remark}
Theorem~\ref{thm:IntroMain} is an algebraic expression of an important ``compatibility between summands'' property of the bordered HFK bimodules. Indeed, like the general strands algebras $\A(\Zc)$ of bordered Heegaard Floer homology, Ozsv{\'a}th--Szab{\'o}'s bordered HFK algebras have a direct sum decomposition indexed by $\Z$ (in Heegaard diagram terms this index describes occupancy number, while representation-theoretically it encodes a $\gl(1|1)$ weight space decomposition). The $A_{\infty}$ bimodules for tangles respect this decomposition, and there is a certain compatibility between the bimodule summands for different $k$. In \cite{OSzNew}, this compatibility is encoded in a graph from which one can define all summands of the bimodules. Because of how the 2-action bimodules $\E$ interact with the index of the direct sum decomposition, Theorem~\ref{thm:IntroMain} is a more algebraic way to formulate this compatibility.

In \cite{OSzNew}, this compatibility is the key ingredient in the ``global extension'' of the two-strand crossing bimodules to bimodules, over larger algebras, for $n$ strands with one crossing between two adjacent strands (this extension is necessary when using the theory of \cite{OSzNew} to compute HFK for knots). The global extension is one of the most technical parts of \cite{OSzNew}; the main hoped-for application of the results of this paper is a more algebraic treatment of the global extension, based on higher representation theory.
\end{remark}

\begin{remark}
The 1-morphism structure of Theorem~\ref{thm:IntroMain} can be interpreted as an instance of an extra layer of the connection between higher representation theory and cornered Heegaard Floer homology, beyond what was explored in \cite{ManionRouquier}. This extra layer involves 3-manifolds, not just 1- and 2-manifolds, and begins to relate to the parts of cornered Heegaard Floer homology that use holomorphic disk counts and domains in Heegaard diagrams with corners. Generalizing from Theorem~\ref{thm:IntroMain}, there should be a general family of Heegaard diagrams (with the diagrams underlying the bordered HFK bimodules as special cases) whose bimodules can be upgraded to 1-morphisms of 2-representations, and the data needed for this upgrade should come from counting holomorphic disks whose domains have positive multiplicities at the corners of the Heegaard diagram.
\end{remark}

\begin{remark}
This paper is focused on the local two-strand aspects of bordered HFK, since these are the elementary building blocks to which one wants to apply a global extension procedure to obtain $n$-strand tangle invariants. One could also ask whether the globally-extended $n$-strand tangle bimodules of bordered HFK give 1-morphisms of 2-representations of $\U$; we expect this to be true. Furthermore, the local bimodules considered here are adapted to two strands pointing in the same direction (downwards, in the conventions of \cite{OSzNew}). For strands with other orientations, one has a choice of more elaborate theories from \cite{OSzNew,OSzNewer,OSzHolo}, some involving curved dg algebras. We expect that the bimodules of these more elaborate theories also give 1-morphisms of 2-representations of $\U$, once (e.g.) 2-representations are appropriately defined on the curved dg algebras.
\end{remark}

\begin{remark}\label{rem:LargerAlgebras}
Since it follows from \cite{LP, MMW2} that the local Ozsv{\'a}th--Szab{\'o} algebras appearing in this paper are quasi-isomorphic to certain (larger) dg strands algebras $\A(\Zc)$, it is natural to ask whether there are bimodules corresponding to $\Pc$ and $\Nc$ over the larger algebras, and if so, whether these bimodules give 1-morphisms between the 2-representation structures on $\A(\Zc)$ defined directly in \cite{ManionRouquier}. The answer in both cases appears to be ``yes;'' the authors of \cite{MMW2} hope to address this question in work in preparation.
\end{remark}

\subsection*{Organization}

In Section~\ref{sec:BorderedAlgebra} we review algebraic definitions from bordered Heegaard Floer homology, including a matrix-based notation from \cite{ManionTrivalent} that will be useful here. In Section~\ref{sec:BorderedHFK} we review what we need from Ozsv{\'a}th--Szab{\'o}'s theory of bordered HFK. In Section~\ref{sec:HigherReps} we review the relevant input from higher representation theory and define 2-actions of $\U$ on the local bordered HFK algebras. In Section~\ref{sec:P1Morphism} we show that Theorem~\ref{thm:IntroMain} holds for Ozsv{\'a}th--Szab{\'o}'s local positive-crossing bimodule $\Pc$, and in Section~\ref{sec:N1Morphism} we do the same for the local negative-crossing bimodule $\Nc$.

\subsection*{Acknowledgments}

The second author would like to thank Zolt{\'a}n Szab{\'o} for many useful conversations over the years related to bordered HFK and the topics of this paper. A.M. is partially supported by NSF grant DMS-2151786.

\section{Bordered algebra}\label{sec:BorderedAlgebra}

\subsection{\texorpdfstring{$DA$ bimodules}{DA bimodules}}\label{sec:DABimod}

We will work with $DA$ bimodules, as defined by Lipshitz--Ozsv{\'a}th--Thurston \cite[Section 2.2.4]{LOTBimodules}, over associative algebras. We will assume that these associative algebras $\A$ are defined over a field $k$ of characteristic $2$ and come equipped with a finite collection of orthogonal idempotents $\{I_1, \ldots, I_n\}$ such that $I_1 + \cdots + I_n = 1$. We will refer to the $I_j$ as distinguished idempotents. 

\begin{remark}
An equivalent perspective is to view $\A$ as a $k$-linear category with objects $\{I_1, \ldots, I_n\}$.
\end{remark}

For such an algebra $\A$, we will let $\I_{\A}$ denote the ring of idempotents of $\A$, i.e. a finite direct product of copies of $k$ (one for each idempotent $I_j$), viewed as a subalgebra of $\A$.

We will also assume that $\A$ is equipped with two $\Z$-gradings which we will call the intrinsic and homological gradings; we let $[1]$ denote an upward shift by $1$ in the homological grading (we use upward rather than downward shifts because, following the conventions of \cite{LOTBimodules,OSzNew}, we use differentials that decrease the homological grading by $1$).

\begin{definition}\label{def:DABimod}
Let $\A$ and $\B$ be graded associative algebras over a field $k$ of characteristic $2$. A $DA$ bimodule over $(\A,\B)$ is given by the data $(X, (\delta^1_i)_{i=1}^{\infty})$ where $X$ is a $\Z \oplus \Z$-graded bimodule over $(\I_{\A}, \I_{\B})$ and, for $i \geq 1$,
\[
\delta^1_i \colon X \otimes \B[1]^{\otimes(i-1)} \to \A[1] \otimes X
\]
(tensor products are over $\I_{\A}$ or $\I_{\B}$ as appropriate) is a bidegree-preserving morphism of bimodules over $(\I_{\A},\I_{\B})$ such that the $DA$ bimodule relations are satisfied, i.e. such that
\begin{align*}
&\sum_{j_1 + j_2 = i+1} (\mu_{\A} \otimes \id_X) \circ (\id_{\A} \otimes \delta^1_{j_1}) \circ (\delta^1_{j_2} \otimes \id_{\B^{\otimes (j_1 - 1)}}) \\
&+ \sum_{j=1}^{i-2} \delta^1_{i-1} \circ (\id_{\B^{\otimes(j-1)}} \otimes \mu_{\B} \otimes \id_{\B^{\otimes(i-j-2)}}) \\
&= 0
\end{align*}
for all $i \geq 1$, where $\mu_{\A}$ and $\mu_{\B}$ are the multiplication operations on $\A$ and $\B$.
\end{definition}

We will often refer to $(X, (\delta^1_i)_{i=1}^{\infty})$ simply as $X$. We say that $X$ is strictly unital if $\delta^1_2(x,1) = 1 \otimes x$ for all $x \in X$ and $\delta^1_i(x,b_1,\ldots,b_{i-1}) = 0$ if $i > 2$ and any $b_j$ is in the idempotent ring $\I_{\B}$.

If we have a $k$-basis for $X$ and $x,x'$ are basis elements with $a \otimes x'$ appearing as a nonzero term of $\delta^1_i(x \otimes b_1 \otimes \cdots \otimes b_{i-1})$ (where $a \in \A$ and $b_1, \ldots, b_{i-1} \in \B$), we will sometimes depict the situation using a ``$DA$ module operation graph'' as in \cite[Definition 2.2.45]{LOTBimodules}. See Figure~\ref{fig:BasicOperationTree} for an example. In this notation, the $DA$ bimodule relations are shown in Figure~\ref{fig:DABimoduleRels}.

\begin{figure}
    \centering
    \includegraphics[scale=0.43]{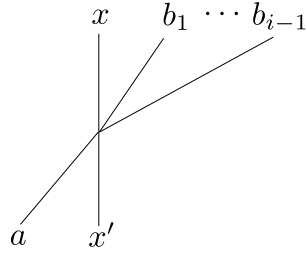}
    \caption{A $DA$ module operation graph showing a term $a \otimes x'$ of the action of $\delta^1_i$ on $x \otimes b_1 \otimes \cdots \otimes b_{i-1}$.}
    \label{fig:BasicOperationTree}
\end{figure}

\begin{figure}
    \centering
    \includegraphics[scale=0.43]{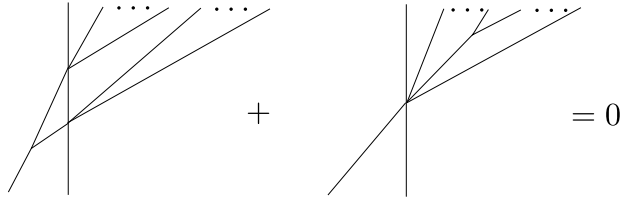}
    \caption{The DA bimodule relations in Definition~\ref{def:DABimod}.}
    \label{fig:DABimoduleRels}
\end{figure}

\begin{remark}
For all $DA$ bimodules $(X, (\delta^1_i)_{i=1}^{\infty})$ considered in this paper, $X$ will be finite-dimensional over $k$, as well as left and right bounded in the sense of \cite[Definition 2.2.46]{LOTBimodules}.
\end{remark}

\begin{remark}
If $X$ is a $DA$ bimodule over $(\A,\B)$, then $\A \otimes_{\I_{\A}} X$ is an $A_{\infty}$ bimodule over $(\A,\B)$ such that the left action of $\A$ has no higher $A_{\infty}$ terms and such that, as a left $\A$-module, $X$ is a direct sum of projective modules $\A \cdot I$ for distinguished idempotents $I$ of $\A$ (disregarding the differential). One can think of the definition of $DA$ bimodule as a convenient way of specifying and reasoning about such $A_{\infty}$ bimodules.
\end{remark}

\subsection{The box tensor product}\label{sec:BoxTensor}

Let $\A,\B,\Cc$ be associative algebras as in Section~\ref{sec:DABimod} and let $X$ and $Y$ be $DA$ bimodules over $(\A,\B)$ and $(\B,\Cc)$ respectively. Assuming $X$ is left bounded or $Y$ is right bounded, Lipshitz--Ozsv{\'a}th--Thurston define a $DA$ bimodule $X \boxtimes Y$ in \cite[Section 2.3.2]{LOTBimodules}.

\begin{definition}
As a bimodule over $(\I_{\A},\I_{\Cc})$, $X \boxtimes Y$ is defined to be $X \otimes_{\I_{\B}} Y$. For $i \geq 1$, the $DA$ bimodule operation $\delta^{\boxtimes,1}_i$ on $X \boxtimes Y$ is defined in terms of the operations $\delta^{X,1}_*$ on $X$ and $\delta^{Y,1}_*$ on $Y$ by
\begin{align*}
\delta^{\boxtimes,1}_i = \sum_{j \geq 0} \sum_{i_1 + \cdots + i_j = i + j - 1} &(\delta^{X,1}_j \otimes \id_Y) \circ (\id_X \otimes \id_{\B^{\otimes(j-1)}} \otimes \delta^{Y,1}_{i_j}) \\
&\circ (\id_X \otimes \id_{\B^{\otimes(j-2)}} \otimes \delta^{Y,1}_{i_{j-1}} \otimes \id_{\A^{\otimes(i_j - 1)}}) \\
& \circ \cdots \circ (\id_X \otimes \delta^{Y,1}_{i_1} \otimes \id_{\A^{\otimes(i_2 + \cdots + i_j - j + 1)}}).
\end{align*}
\end{definition}

In terms of $DA$ module operation graphs, the general pattern for the operation $\delta^{\boxtimes,1}_i$ on $X \boxtimes Y$ is shown in Figure~\ref{fig:BoxTensorTree}.

\begin{figure}
    \centering
    \includegraphics[scale=0.43]{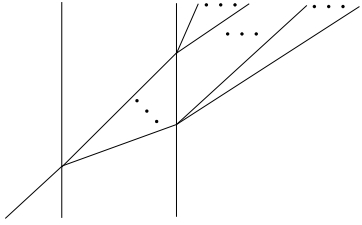}
    \caption{The general pattern for the operation $\delta^{\boxtimes,1}_i$ on $X \boxtimes Y$.}
    \label{fig:BoxTensorTree}
\end{figure}

\begin{remark}
By \cite[Proposition 2.3.10]{LOTBimodules}, if $X$ and $Y$ are both left bounded then so is $X \boxtimes Y$.
\end{remark}

\begin{remark}
Assuming suitable boundedness, the box tensor product $X \boxtimes Y$ is a convenient way of working with the derived tensor product $(\A \otimes_{\I_{\A}} X) \tilde{\otimes}_{\B} (\B \otimes_{\I_{\B}} Y)$; indeed, by \cite[Proposition 2.3.18]{LOTBimodules} we have
\[
\A \otimes_{\I_{\A}} (X \boxtimes Y) \simeq (\A \otimes_{\I_{\A}} X) \tilde{\otimes}_{\B} (\B \otimes_{\I_{\B}} Y)
\]
where $\simeq$ denotes homotopy equivalence of $DA$ bimodules (see \cite[Section 2.2.4]{LOTBimodules}).
\end{remark}

\subsection{Matrix notation}

We will describe $DA$ bimodules using the matrix-based notation of \cite[Section 2.2]{ManionTrivalent}; we recall this notation here. When using this notation to describe a $DA$ bimodule over $(\A,\B)$, it is assumed that $\B$ comes equipped with a choice of $k$-basis such that:
\begin{itemize}
    \item distinguished idempotents of $\B$ are basis elements;
    \item each basis element $b$ satisfies $I \cdot b \cdot I' = b$ for unique distinguished idempotents $I$ of $\A$ and $I'$ of $\B$ (called the left and right idempotents of $b$ respectively) with $\tilde{I} \cdot b \cdot \tilde{I}' = 0$ whenever $\tilde{I}, \tilde{I}'$ are distinguished idempotents of $\A$ and $\B$ with $\tilde{I} \neq I$ or $\tilde{I}' \neq I'$;
    \item each basis element of $\B$ is homogeneous with respect to the bigrading.
\end{itemize}

\begin{definition}
To specify a $DA$ bimodule $(X, (\delta^1_i)_{i=1}^{\infty})$ over $(\A,\B)$ (finite-dimensional over $k$), we specify two matrices, a \emph{primary matrix} and a \emph{secondary matrix}.
\begin{itemize}
    \item The primary matrix is a set-valued matrix (each entry is a finite set with a $\Z \oplus \Z$-bidegree specified for each element) with columns indexed by the distinguished idempotents of $\B$ and rows indexed by the distinguished idempotents of $\A$. Given such a matrix, the bimodule $X$ over $(\I_{\A},\I_{\B})$ is taken to have a $k$-basis given by the union of the sets in each entry (with each basis element given its specified bidegree). More specifically, the left action of $\I_{\A}$ and right action of $\I_{\B}$ are fixed by saying that, for distinguished idempotents $I$ of $\A$ and $I'$ of $\B$, the vector space $I \cdot X \cdot I'$ has a basis given by the set in row $I$ and column $I'$. For an element $x$ of this set, we say that $I$ is the left idempotent of $x$ and $I'$ is the right idempotent of $x$.

    \item The secondary matrix is a matrix whose entries are formal sums of expressions $a$ (for $a \in \A$) and $a \otimes (b_1, \ldots, b_{i-1})$ (for $a \in \A$ and each $b_j$ a basis element for $\B$). The sums are allowed to be infinite, but there should be finitely many terms $a$ (without the $\otimes$ symbol) and finitely many terms for each given sequence $(b_1, \ldots, b_{i-1})$. The rows and columns of the secondary matrix are each given by the union of all entries of the primary matrix, in some fixed order. Given such a matrix, the operations $\delta^1_i$ on $X$ are defined as follows for a basis element $x$ of $X$ (a column label of the secondary matrix):
    \begin{itemize}
        \item $\delta^1_1(x)$ is the sum of all elements $a \otimes y$ where $a$ is a term (without the $\otimes$ symbol) of a secondary matrix entry in column $x$ and $y$ is the row label of the entry containing this term;
        
        \item for $i > 1$ and a sequence $(b_1,\ldots,b_{i-1})$ of basis elements of $\B$, $\delta^1_i(x \otimes b_1 \otimes \cdots \otimes b_{i-1})$ is the sum of all elements $a \otimes y$ where $a \otimes (b_1, \ldots, b_{i-1})$ is a term of a secondary matrix entry in column $x$ and $y$ is the row label of the entry containing this term.
    \end{itemize}
\end{itemize}
\end{definition}

An example of a $DA$ bimodule specified by primary and secondary matrices can be found in Definition~\ref{def:PBimodule} below. We use the following conventions:

\begin{convention}
If indices such as $k$ or $l$ appear in entries of the secondary matrix, we take an infinite sum over all $k \geq 0$ or $l \geq 0$ unless otherwise specified.
\end{convention}

\begin{convention}\label{conv:UnitsOmitted}
When using matrix notation to specify a strictly unital $DA$ bimodule, the above rules would say that in each diagonal entry of the secondary matrix (corresponding to an entry $x$ of the primary matrix), there is a term $I \otimes I'$ where $I$ and $I'$ are the left and right idempotents of $x$ respectively (it should also be the case that no basis element $b_j$ appearing in an entry $a \otimes (b_1, \ldots, b_{i-1})$ is a distinguished idempotent). However, we will omit the terms $I \otimes I'$ when we write the secondary matrix.
\end{convention}

If the primary or secondary matrix has block form, we will often give each block separately.

\begin{remark}
One advantage of this matrix-based notation is that the $DA$ bimodule relations can be checked using linear-algebraic manipulations. Indeed, to check the $DA$ bimodule relations, one forms two new matrices from the secondary matrix. The first matrix, which we will call the ``squared secondary matrix,'' is obtained by multiplying the secondary matrix by itself. When doing so, one will need to take products of secondary matrix entries; these products are defined by
\begin{itemize}
    \item $a \cdot a' = a' a$
    \item $a \cdot (a' \otimes (b'_1, \ldots, b'_{i-1}) = a'a \otimes (b'_1, \ldots, b'_{i-1})$
    \item $(a \otimes (b_1, \ldots, b_{i-1}) \cdot a' = a'a \otimes (b_1,\ldots, b_{i-1})$
    \item $(a \otimes (b_1, \ldots, b_{i-1}) \cdot (a' \otimes (b'_1, \ldots, b'_{j-1}) = a'a \otimes (b'_1, \ldots, b'_{j-1}, b_1, \ldots, b_{i-1})$
\end{itemize}
The second matrix, which we will call the ``multiplication matrix,'' is obtained by, for each $b_j$ in an entry $a \otimes (b_1, \ldots, b_{i-1})$ and each pair of $\B$-basis elements $(b', b'')$ (neither a distinguished idempotent in the strictly unital case) such that $C b_j$ is a term of the basis expansion of $b' b''$ for some nonzero element $C \in k$, adding the term $Ca \otimes (b_1, \ldots, b_{j-1}, b', b'', b_{j+1}, \ldots, b_{i-1})$ to the corresponding entry of the multiplication matrix.

Once these two matrices are formed, the $DA$ bimodule relations amount to saying that the squared secondary matrix and the multiplication matrix sum to zero.
\end{remark}

\subsection{Box tensor products in matrix notation}\label{sec:MatrixBoxTensor}

Suppose we have $DA$ bimodules $X$ over $(\A,\B)$ and $Y$ over $(\B,\Cc)$ as in Section~\ref{sec:BoxTensor}. To specify $X \boxtimes Y$ in matrix notation, one can do the following manipulations:
\begin{itemize}
    \item The primary matrix for $X \boxtimes Y$ is the matrix product of the primary matrix for $X$ (on the left) and the primary matrix for $Y$ (on the right). When multiplying two entries of these primary matrices, one uses the Cartesian product of sets, and when adding these products together, one uses the disjoint union.
    
    \item Let $(x,y)$ and $(x',y')$ be two elements of the primary matrix for $X \boxtimes Y$. To obtain the secondary matrix element in row $(x',y')$ and column $(x,y)$, there are two cases to consider:
    \begin{itemize}
        \item For entries $a$ (with no $\otimes$ symbol) in row $x'$ and column $x$ of the secondary matrix for $X$, if $y = y'$ then add an entry $a$ to the secondary matrix for $X \boxtimes Y$ in row $(x',y')$ and column $(x,y)$. If $y \neq y'$, do not add such an entry.
        
        \item For entries $a \otimes (b_1, \ldots, b_{i-1})$ in row $x'$ and column $x$ of the secondary matrix for $X$, look for all sequences $(y = y_1, y_2, \ldots, y_i = y')$ of primary matrix entries for $Y$ such that, for $1 \leq j \leq i-1$, there is a term $b \otimes (c^j_1, \ldots, c^j_{m_j - 1})$ in row $y_{j+1}$ and column $y_j$ of the secondary matrix for $Y$ such that $C_j b_j$ is a term of the basis expansion of $b$ for some nonzero $C_j \in k$. For all such sequences $(y_1, \ldots, y_i)$ and all such choices of terms $b \otimes (c^j_1, \ldots, c^j_{m_j - 1})$, add an entry
        \[
        C_1 \cdots C_{i-1} a \otimes (c^1_1, \ldots, c^1_{m_1 - 1}, \ldots, c^{i-1}_1, \ldots, c^{i-1}_{m_{i-1} - 1})
        \]
        to the secondary matrix of $X \boxtimes Y$ in row $(x',y')$ and column $(x,y)$.
    \end{itemize}
\end{itemize}

\section{Bordered HFK}\label{sec:BorderedHFK}

\subsection{Algebras}

We now review Ozsv{\'a}th--Szab{\'o}'s algebra $\displaystyle \B(2) = \bigoplus_{k=0}^3 \B(2,k)$ from \cite[Section 3.2]{OSzNew}, which is an algebra over $\F_2$. 

\begin{definition}
The algebra $\B(2,0)$ is $\F_2$. The algebra $\B(2,1)$ is the path algebra of the quiver shown in Figure~\ref{fig:Quiver} modulo the relations $[R_i, U_j] = 0$, $[L_i,U_j] = 0$, $R_i L_i = U_i$, $L_i R_i = U_i$, $R_1 R_2 = 0$, $L_2 L_1 = 0$, $U_2 = 0$ at the leftmost node, and $U_1 = 0$ at the rightmost node. 

The algebra $\B(2,2)$ is the path algebra of the quiver shown in Figure~\ref{fig:Quiver2} modulo the relations $[R_i U_j] = 0$, $[L_i,U_j] = 0$, $R_i L_i = U_i$, and $L_i R_i = U_i$. The algebra $\B(2,3)$ is $\F_2[U_1, U_2]$. We set $\displaystyle \B(2) = \bigoplus_{k=0}^3 \B(2,k)$.
\end{definition}

\begin{figure}
    \centering
    \includegraphics[scale=0.9]{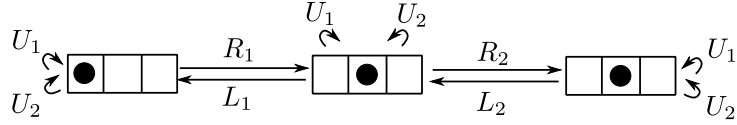}
    \caption{The quiver for $\B(2,1)$.}
    \label{fig:Quiver}
\end{figure}

\begin{figure}
    \centering
    \includegraphics[scale=0.9]{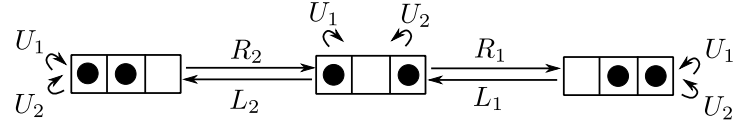}
    \caption{The quiver for $\B(2,2)$.}
    \label{fig:Quiver2}
\end{figure}

Our definition matches Ozsv{\'a}th--Szab{\'o}'s by \cite[Theorem 1.1]{MMW1}; also see \cite[Figure 10]{OSzNew} for $\B(2,1)$, although in this figure Ozsv{\'a}th--Szab{\'o} leave out some of the relations. We define an intrinsic grading on $\B(2)$ by setting $\deg(R_i) = \deg(L_i) = 1$ and $\deg(U_i) = 2$; this grading is twice Ozsv{\'a}th--Szab{\'o}'s single Alexander grading (the doubling is related to the expression $t = q^2$ when obtaining the Alexander polynomial from representations of $U_q(\gl(1|1))$). We define the homological grading to be identically zero on the generators of $\B(2)$.

The algebras $\B(2,1)$ and $\B(2,2)$ each have three distinguished idempotents given by the length-zero paths at each node. Ordering the nodes from left to right and following Ozsv{\'a}th--Szab{\'o}'s notation, for $\B(2,1)$ we can call these idempotents $I_0$, $I_1$, and $I_2$. For $\B(2,2)$ we can call them $I_{01}$, $I_{02}$, and $I_{12}$. The unique nonzero element of $\B(2,0)$ is its distinguished idempotent and we can call it $I_{\varnothing}$; for $\B(2,3)$ the distinguished idempotent is $1 \in \F_2[U_1,U_2]$ and we can call it $I_{012}$.

To avoid subscripts as much as possible, we will relabel these idempotents as follows:
\[
\varnothing := I_{\varnothing},
\]
\[
A := I_0, \quad B := I_1, \quad C := I_2,
\]
\[
AB := I_{01}, \quad AC := I_{02}, \quad BC := I_{12},
\]
\[
ABC := I_{012}.
\]
To clarify the conventions: in Figure~\ref{fig:Quiver} the left and right idempotents of $R_1$ are $A$ and $B$ respectively, while in Figure~\ref{fig:Quiver2} the left and right idempotents of $R_1$ are $AC$ and $BC$ respectively.

The following proposition can be deduced from the definition of $\B(2)$.

\begin{proposition}
A $\F_2$-basis for $\B(2,1)$ is given by 
\[
\{U_1^k (A), \,\, U_1^k (B), \,\, U_2^k(B), \,\, U_2^k(C), \,\, R_1U_1^k, \,\, L_1 U_1^k, \,\, R_2 U_2^k, \,\, L_2 U_2^k\}
\]
($k$ runs over all integers $\geq 0$). A $\F_2$-basis for $\B(2,2)$ is given by 
\[
\{U_1^k U_2^k (AB), \,\, U_1^k U_2^l (AC), \,\, U_1^k U_2^l (BC), \,\, R_1 U_1^k U_2^l, \,\, L_1 U_1^k U_2^l, \,\, R_2 U_1^k U_2^l, \,\, L_2 U_1^k U_2^l)
\]
($k$ and $l$ run over all integers $\geq 0$).
\end{proposition}

The algebra $\B(2,0) = \F_2$ has a unique $\F_2$-basis, and for $\B(2,3)$ we use the basis of monomials $U_1^k U_2^l$ for $k, l \geq 0$.

\subsection{Bimodules}

Next we review, in matrix notation, Ozsv{\'a}th--Szab{\'o}'s $DA$ bimodules $\Pc$ and $\Nc$ over $\B(2)$. One thinks of these bimodules as being associated to two-strand tangles consisting of a single positive crossing and a single negative crossing respectively and containing the minimal amount of data necessary to build the bimodules for $n$-strand single-crossing tangles. They can be obtained by counting holomorphic disks in the Heegaard diagrams shown in Section~\ref{sec:HeegaardDiagrams} below.

\subsubsection{The bimodule $\Pc$}

This bimodule is defined in \cite[Section 5.1]{OSzNew}; here we translate Ozsv{\'a}th--Szab{\'o}'s definition into matrix notation. 

\begin{definition}\label{def:PBimodule}
The primary matrix for $\mathcal{P}$ has rows and columns indexed by the distinguished idempotents
\[
\varnothing, A, B, C, AB, AC, BC, ABC
\]
of $\B(2)$. The matrix has block-diagonal form with blocks specified by the following matrices:
\[
\kbordermatrix{
 & \varnothing \\
 \varnothing & {_{\varnothing}}S_{\varnothing}
},
\]
\[
\kbordermatrix{
 & A & B & C \\
A & \{{_A}S_A\} & \varnothing & \varnothing \\
B & \{{_B}W_A\} & \{{_B}N_B\} & \{{_B}E_C\} \\
C & \varnothing & \varnothing & \{{_C}S_C\}
},
\]
\[
\kbordermatrix{
 & AB & AC & BC \\
AB & \{{_{AB}}N_{AB}\} & \{{_{AB}}E_{AC}\} & \varnothing \\
AC & \varnothing & \{{_{AC}}S_{AC}\} & \varnothing \\
BC & \varnothing & \{{_{BC}}W_{AC}\} & \{{_{BC}}N_{BC}\}
},
\]
\[
\kbordermatrix{
 & ABC \\
ABC & \{{_{ABC}}N_{ABC}\}
}.
\]
Below we will abuse notation slightly and omit the braces $\{\}$, writing e.g. ${_A}S_A$ instead of $\{{_A}S_A\}$. The secondary matrix for $\Pc$ has a corresponding block-diagonal form; the blocks are 
\[
   \kbordermatrix{
    & {_\varnothing}S_{\varnothing} \\
   {_\varnothing}S_{\varnothing} & 0
    },
\]
\[
\resizebox{\textwidth}{!}{
\kbordermatrix{
   & {_A}S_{A} & {_B}W_{A} & {_B}N_{B} & {_B}E_{C} & {_C}S_{C}\\
{_A}S_{A} & 0 & L_1  & 0 & 0 & 0 \\
 {_B}W_{A} & 0 & U_2^{k+1} \otimes U_1^{k+1} & U_2^{k+1} \otimes L_1 U_1^k & 0 & L_2U_2^k \otimes (L_2, L_1U_1^k)\\
    {_B}N_{B} & R_1 U_1^k \otimes (R_1, U_2^{k+1}) & U_2^{k} \otimes R_1U_1^k & U_2^{k+1} \otimes U_1^{k+1} + U_1^{k+1} \otimes U_2^{k+1} & U_1^{k} \otimes L_2U_2^k & L_2 U_2^k \otimes (L_2, U_1^{k+1}) \\
   {_B}E_{C} & R_1 U_1^k \otimes (R_1, R_2 U_2^k)  & 0 & U_1^{k+1} \otimes R_2 U_2^k & U_1^{k+1}\otimes U_2^{k+1} & 0\\
    {_C}S_{C} & 0 & 0  & 0 & R_2 & 0
    }
},
\]
\[
\resizebox{\textwidth}{!}{
    \kbordermatrix{
   & {_{AB}}N_{AB} & {_{AB}}E_{AC} & {_{AC}}S_{AC} & {_{BC}}W_{AC} & {_{BC}}N_{BC}\\
    {_{AB}}N_{AB}& U_1^lU_2^k\otimes U_1^kU_2^l & U_1^k\otimes L_2U_2^k & *_1 & L_1L_2U_2^k \otimes L_2U_1^{k+1} & L_1 L_2 U_1^l U_2^k \otimes L_1 L_2 U_1^k U_2^l\\
   {_{AB}}E_{AC} & U_1^{l+1}U_2^k\otimes R_2U_1^kU_2^l & U_1^{k+1}\otimes U_2^{k+1} & *_2 & L_1L_2U_2^k \otimes U_1^{k+1} & L_1L_2U_1^{l}U_2^k \otimes L_1U_1^kU_2^l\\
   {_{AC}}S_{AC} & 0 & R_2 & 0 & L_1 & 0 \\
   {_{BC}}W_{AC} & R_2 R_1 U_1^lU_2^{k}\otimes R_2U_1^kU_2^l & R_2R_1U_1^k\otimes U_2^{k+1} & *_3 & U_2^{k+1} \otimes U_1^{k+1} & U_1^lU_2^{k+1}\otimes L_1U_1^kU_2^l\\
   {_{BC}}N_{BC} & R_2R_1U_1^lU_2^k \otimes  R_2R_1U_1^kU_2^l & R_2R_1U_1^k\otimes R_1U_2^{k+1} & *_4 & U_2^k\otimes R_1U_1^k & U_1^lU_2^k \otimes U_1^kU_2^l
   }
},
\]
\[
    \kbordermatrix{
    & {_{ABC}}N_{ABC}\\
  {_{ABC}}N_{ABC} & U_1^l U_2^k \otimes U_1^k U_2^l}.
\]
The entries $*_i$ for $1 \leq i \leq 4$ are specified below; also, in any entry of the form $U_1^l U_2^k \otimes U_1^k U_2^l$, we disallow $(k,l) = (0,0)$ to match Convention~\ref{conv:UnitsOmitted}. The entry $*_1$ in column $_{AC}S_{AC}$ and row $_{AB}N_{AB}$ is:
\begin{align*}
& L_2U_1^tU_2^n \otimes (U_1^{n+1}, L_2U_2^t) \quad (0 \leq n < t)\\
&+ L_2U_1^tU_2^n \otimes (R_1U_1^n, L_1L_2U_2^t) \quad (0 \leq n < t) \\
&+ L_2U_1^tU_2^n \otimes (L_2U_1^{n+1}, U_2^t) \quad (0 \leq n < t) \\
&+ L_2U_1^tU_2^n \otimes (L_2U_2^t, U_1^{n+1}) \quad (1 \leq t \leq n) \\
&+ L_2U_1^tU_2^n \otimes (U_2^t, L_2U_1^{n+1}) \quad (1 \leq t \leq n) \\
&+ L_2U_1^tU_2^n \otimes (R_1U_2^t, L_1L_2U_1^n) \quad (1 \leq t \leq n) \\
&+ L_2 U_2^n \otimes (L_2, U_1^{n+1}) \quad (0 \leq n)
\end{align*}
The entry $*_2$ in column $_{AC}S_{AC}$ and row $_{AB}E_{AC}$ is:
\begin{align*}
&L_2U_1^tU_2^n \otimes (U_1^{n+1}, U_2^t)\quad  (0 \leq n < t) \\
&+ L_2U_1^tU_2^n \otimes (R_1U_1^n, L_1U_2^t) \quad  (0 \leq n < t)\\
&+ L_2U_1^tU_2^n \otimes (L_2 U_1^{n+1}, R_2U_2^{t-1}) \quad  (0 \leq n < t)\\
&+ L_2U_1^tU_2^n \otimes (U_2^t, U_1^{n+1}) \quad  (1 \leq t \leq n) \\
&+ L_2U_1^tU_2^n \otimes (R_1U_2^t, L_1U_1^n) \quad  (1 \leq t \leq n)  \\
&+ L_2U_1^tU_2^n \otimes (L_2 U_2^{t-1}, R_2U_1^{n+1}) \quad  (1 \leq t \leq n) 
\end{align*}
The entry $*_3$ in column $_{AC}S_{AC}$ and row $_{BC}W_{AC}$ is:
\begin{align*}
    &R_1U_1^tU_2^n \otimes (U_2^{t+1}, U_1^n)  \quad  (0 \leq t < n)  \\
    &+ R_1U_1^tU_2^n \otimes (L_2U_2^t, R_2U_1^n) \quad  (0 \leq t < n)\\
    &+ R_1U_1^tU_2^n \otimes (R_1 U_2^{t+1}, L_1 U_1^{n-1}) \quad  (0 \leq t < n) \\
    &+ R_1U_1^tU_2^n \otimes (U_1^n, U_2^{t+1}) \quad  (1 \leq n \leq t) \\
    &+ R_1U_1^tU_2^n \otimes (L_2U_1^n, R_2U_2^t) \quad  (1 \leq n \leq t)\\
    &+ R_1U_1^tU_2^n \otimes (R_1U_1^{n-1}, L_1U_2^{t+1})\quad  (1 \leq n \leq t)
\end{align*}
The entry $*_4$ in column $_{AC}S_{AC}$ and row $_{BC}N_{BC}$ is:
\begin{align*}
    &R_1U_1^tU_2^n \otimes (U_2^{t+1}, R_1U_1^n) \quad (0 \leq t < n) \\
    &+  R_1U_1^tU_2^n \otimes (L_2U_2^t, R_2R_1U_1^n) \quad (0 \leq t < n) \\
    &+  R_1U_1^tU_2^n \otimes (R_1U_2^{t+1}, U_1^n) \quad (0 \leq t < n)\\
    &+  R_1U_1^tU_2^n \otimes (R_1U_1^n, U_2^{t+1}) \quad (1 \leq n \leq t)\\
    &+  R_1U_1^tU_2^n \otimes (U_1^n, R_1U_2^{t+1}) \quad (1 \leq n \leq t)\\
    &+  R_1U_1^tU_2^n \otimes (L_2U_1^n, R_2R_1U_2^t) \quad (1 \leq n \leq t)\\
    &+  R_1U_1^t \otimes (R_1, U_2^{t+1})\quad (0 \leq t)
\end{align*}
\end{definition}

\subsubsection{The bimodule $\Nc$}

The bimodule $\Nc$ is defined in \cite[Section 5.5]{OSzNew} using a symmetry relationship with $\Pc$. Explicitly, $\Nc$ has the same primary matrix as $\Pc$. The blocks of the secondary matrix of $\Nc$ are
\[
   \kbordermatrix{
    & {_\varnothing}S_{\varnothing} \\
   {_\varnothing}S_{\varnothing} & 0
    },
\]
\[
\resizebox{\textwidth}{!}{
\kbordermatrix{
   & {_A}S_{A} & {_B}W_{A} & {_B}N_{B} & {_B}E_{C} & {_C}S_{C}\\
{_A}S_{A} & 0 & 0 & L_1 U_1^k \otimes (U_2^{k+1}, L_1) & L_1 U_1^k \otimes (L_2 U_2^k, L_1) & 0 \\
 {_B}W_{A} & R_1 & U_2^{k+1} \otimes U_1^{k+1} & U_2^k \otimes L_1 U_1^k & 0 & 0 \\
    {_B}N_{B} & 0 & U_2^{k+1} \otimes R_1 U_1^k & U_2^{k+1} \otimes U_1^{k+1} + U_1^{k+1} \otimes U_2^{k+1} & U_1^{k+1} \otimes L_2 U_2^k & 0 \\
   {_B}E_{C} & 0 & 0 & U_1^k \otimes R_2 U_2^k & U_1^{k+1} \otimes U_2^{k+1} & L_2 \\
    {_C}S_{C} & 0 & R_2 U_2^k \otimes (R_1 U_1^k, R_2) & R_2 U_2^k \otimes (U_1^{k+1}, R_2) & 0 & 0
    }
},
\]
\[
\resizebox{\textwidth}{!}{
    \kbordermatrix{
   & {_{AB}}N_{AB} & {_{AB}}E_{AC} & {_{AC}}S_{AC} & {_{BC}}W_{AC} & {_{BC}}N_{BC}\\
    {_{AB}}N_{AB}& U_1^lU_2^k\otimes U_1^kU_2^l & U_1^{l+1} U_2^k \otimes L_2 U_1^k U_2^l & 0 & L_1 L_2 U_1^l U_2^k \otimes L_2 U_1^k U_2^l & L_1 L_2 U_1^l U_2^k \otimes L_1 L_2 U_1^k U_2^l \\
   {_{AB}}E_{AC} & U_1^k \otimes R_2 U_2^k & U_1^{k+1}\otimes U_2^{k+1} & L_2 & L_1 L_2 U_1^k \otimes U_2^{k+1} & L_1 L_2 U_1^k \otimes L_1 U_2^{k+1} \\
   {_{AC}}S_{AC} & *'_1 & *'_2 & 0 & *'_3 & *'_4 \\
   {_{BC}}W_{AC} & R_2 R_1 U_2^k \otimes R_2 U_1^{k+1} & R_2 R_1 U_2^k \otimes U_1^{k+1} & R_1 & U_2^{k+1} \otimes U_1^{k+1} & U_2^k \otimes L_1 U_1^k \\
   {_{BC}}N_{BC} & R_2 R_1 U_1^l U_2^k \otimes R_2 R_1 U_1^k U_2^l & R_2 R_1 U_1^l U_2^k \otimes R_1 U_1^k U_2^l & 0 & U_1^l U_2^{k+1} \otimes R_1 U_1^k U_2^l & U_1^lU_2^k \otimes U_1^kU_2^l
   }
},
\]
\[
    \kbordermatrix{
    & {_{ABC}}N_{ABC}\\
  {_{ABC}}N_{ABC} & U_1^l U_2^k \otimes U_1^k U_2^l}
\]
where in any entry of the specific form $U_1^l U_2^k \otimes U_1^k U_2^l$ we disallow $(k,l) = (0,0)$ to match Convention~\ref{conv:UnitsOmitted}. The entry $*'_1$ in column ${_{AB}}N_{AB}$ and row ${_{AC}}S_{AC}$ is:
\begin{align*}
    & R_2 U_1^t U_2^n \otimes (R_2 U_2^t, U_1^{n+1}) \quad (0 \leq n < t) \\
    &+ R_2 U_1^t U_2^n \otimes (R_2 R_1 U_2^t, L_1 U_1^n) \quad (0 \leq n < t) \\
    &+ R_2 U_1^t U_2^n \otimes (U_2^t, R_2 U_1^{n+1}) \quad (0 \leq n < t) \\
    &+ R_2 U_1^t U_2^n \otimes (U_1^{n+1}, R_2 U_2^t) \quad (1 \leq t \leq n) \\
    &+ R_2 U_1^t U_2^n \otimes (R_2 U_1^{n+1}, U_2^t) \quad (1 \leq t \leq n) \\
    &+ R_2 U_1^t U_2^n \otimes (R_2 R_1 U_1^n, L_1 U_2^t) \quad (1 \leq t \leq n) \\
    &+ R_2 U_2^n \otimes (U_1^{n+1}, R_2) \quad (0 \leq n) \\
\end{align*}
The entry $*'_2$ in column ${_{AB}}E_{AC}$ and row ${_{AC}}S_{AC}$ is:
\begin{align*}
    & R_2 U_1^t U_2^n \otimes (U_2^t, U_1^{n+1}) \quad (0 \leq n < t) \\
    & R_2 U_1^t U_2^n \otimes (R_1 U_2^t, L_1 U_1^n) \quad (0 \leq n < t) \\
    & R_2 U_1^t U_2^n \otimes (L_2 U_2^{t-1}, R_2 U_1^{n+1}) \quad (0 \leq n < t) \\
    & R_2 U_1^t U_2^n \otimes (U_1^{n+1}, U_2^t) \quad (1 \leq t \leq n) \\
    & R_2 U_1^t U_2^n \otimes (R_1 U_1^n, L_1 U_2^t) \quad (1 \leq t \leq n) \\
    & R_2 U_1^t U_2^n \otimes (L_2 U_1^{n+1}, R_2 U_2^{t-1}) \quad (1 \leq t \leq n) \\
\end{align*}
The entry $*'_3$ in column ${_{BC}}W_{AC}$ and row ${_{AC}}S_{AC}$ is:
\begin{align*}
    & L_1 U_1^t U_2^n \otimes (U_1^n, U_2^{t+1}) \quad (0 \leq t < n) \\
    & L_1 U_1^t U_2^n \otimes (L_2 U_1^n R_2 U_2^t) \quad (0 \leq t < n) \\
    & L_1 U_1^t U_2^n \otimes (R_1 U_1^{n-1}, L_1 U_2^{t+1}) \quad (0 \leq t < n) \\
    & L_1 U_1^t U_2^n \otimes (U_2^{t+1}, U_1^n) \quad (1 \leq n \leq t) \\
    & L_1 U_1^t U_2^n \otimes (L_2 U_2^t, R_2 U_1^n) \quad (1 \leq n \leq t) \\
    & L_1 U_1^t U_2^n \otimes (R_1 U_2^{t+1}, L_1 U_1^{n-1}) \quad (1 \leq n \leq t) \\
\end{align*}
The entry $*'_4$ in column ${_{BC}}N_{BC}$ and row ${_{AC}}S_{AC}$ is:
\begin{align*}
    & L_1 U_1^t U_2^n \otimes (L_1 U_1^n, U_2^{t+1}) \quad (0 \leq t < n) \\
    & L_1 U_1^t U_2^n \otimes (L_1 L_2 U_1^n, R_2 U_2^t) \quad (0 \leq t < n) \\
    & L_1 U_1^t U_2^n \otimes (U_1^n, L_1 U_2^{t+1}) \quad (0 \leq t < n) \\
    & L_1 U_1^t U_2^n \otimes (U_2^{t+1}, L_1 U_1^n) \quad (1 \leq n \leq t) \\
    & L_1 U_1^t U_2^n \otimes (L_1 U_2^{t+1}, U_1^n) \quad (1 \leq n \leq t) \\
    & L_1 U_1^t U_2^n \otimes (L_1 L_2 U_2^t, R_2 U_1^n) \quad (1 \leq n \leq t) \\
    & L_1 U_1^t \otimes (U_2^{t+1}, L_1) \quad (0 \leq t) \\
\end{align*}

The starred terms in row ${_{AC}}S_{AC}$ of middle block of the secondary matrix for $\Nc$, as well as in the column ${_{AC}}S_{AC}$ of the middle block of the secondary matrix for $\Pc$, encode the $A_{\infty}$ terms of the right algebra actions on (the middle summands of) the bimodules. See \cite[Section 2.2.4]{LOTBimodules} for more context on these $A_{\infty}$ structures in general.

The symmetry relationship between $\Pc$ and $\Nc$ described in \cite[Section 5.5]{OSzNew} can be summarized by saying the secondary matrix of $\Nc$ is obtained from that of $\Pc$ by performing the following operations:
\begin{itemize}
    \item Take the transpose of the secondary matrix of $\Pc$.
    \item In each entry, replace $L_i$ with $R_i$ and vice-versa, while reversing the order of multiplication when relevant (so e.g. $L_1 L_2$ becomes $R_2 R_1$)
    \item For any entry $a \otimes (b_1, b_2)$, reverse the order of $b_1$ and $b_2$.
\end{itemize}

\subsubsection{Heegaard diagram origins}\label{sec:HeegaardDiagrams}

We comment briefly here on the Heegaard diagram origins of the $DA$ bimodules $\Pc$ and $\Nc$. Roughly, they can be thought of as $DA$ bimodules associated to the bordered sutured Heegaard diagrams shown in Figure~\ref{fig:HeegaardDiagPositive} and Figure~\ref{fig:HeegaardDiagNegative} respectively. A detailed study of the relationship of the algebraically defined bimodules $\Pc$ and $\Nc$ to the holomorphic geometry associated with these diagrams can be found in \cite{OSzHolo}, although in that paper Ozsv{\'a}th--Szab{\'o} do not use the language of bordered sutured Heegaard Floer homology.

\begin{figure}
    \centering
    \includegraphics[scale=0.7]{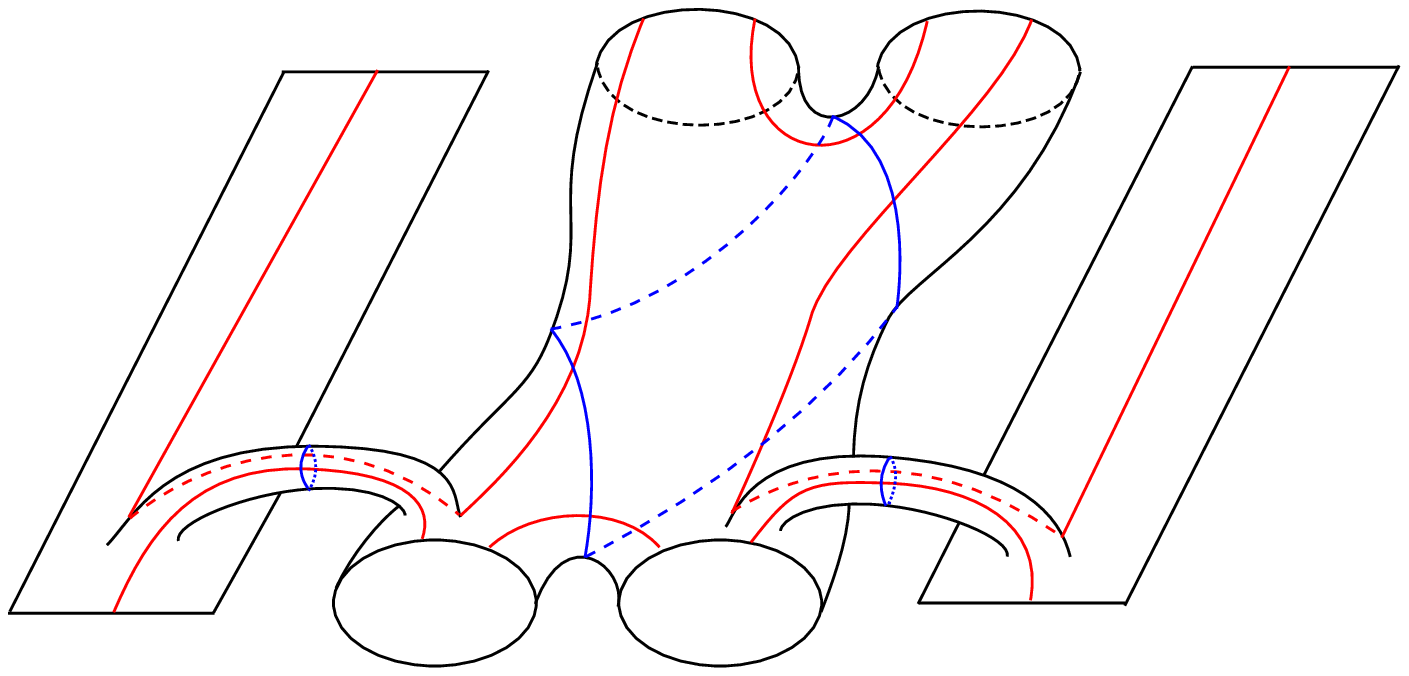}
    \caption{The bordered sutured Heegaard diagram for $\Pc$.}
    \label{fig:HeegaardDiagPositive}
\end{figure}

\begin{figure}
    \centering
    \includegraphics[scale=0.7]{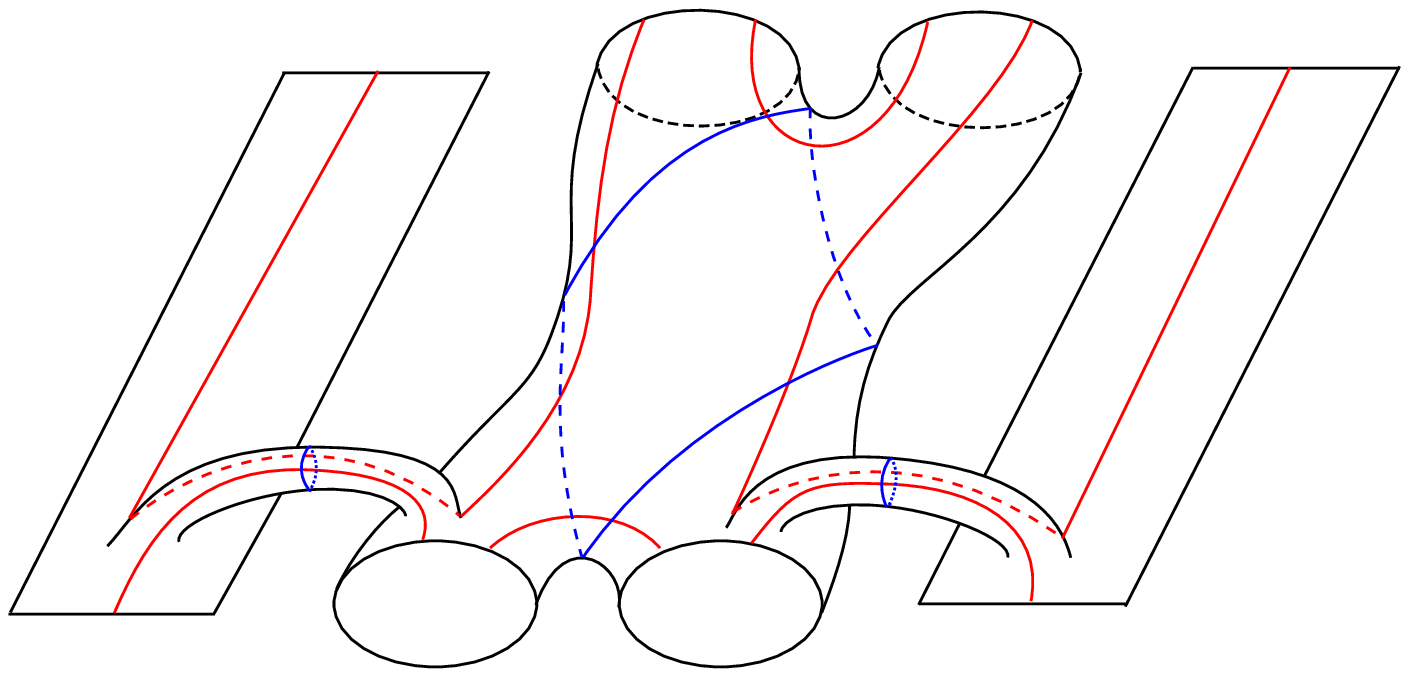}
    \caption{The bordered sutured Heegaard diagram for $\Nc$.}
    \label{fig:HeegaardDiagNegative}
\end{figure}

\begin{remark}
The diagrams in Figure~\ref{fig:HeegaardDiagPositive} and Figure~\ref{fig:HeegaardDiagNegative} do not satisfy all the hypotheses necessary to be covered by Lipshitz--Ozsv{\'a}th--Thurston's results in \cite{LOTBimodules} or Zarev's results in \cite{Zarev}; Ozsv{\'a}th--Szab{\'o} show in \cite{OSzHolo} that they can still be analyzed using a generalization of the analytic setup of bordered or bordered sutured Heegaard Floer homology. However, a more literal generalization of these theories would yield bimodules over the larger dg algebras of \cite{LP,MMW2} rather than over the associative algebra $\B(2)$. The second author, with Marengon and Willis, hope to address this difference in future work, defining $DA$ bimodules over the larger dg algebras and relating them to $\Pc$ and $\Nc$.
\end{remark}

\section{Higher representations}\label{sec:HigherReps}

\subsection{General setup}

We now briefly review how higher representation theory interacts with bordered Heegaard Floer homology, as discussed in more generality in \cite{ManionRouquier}.

\subsubsection{Monoidal category}

The following differential monoidal category $\U$ was defined in \cite{KhovOneHalf}, and 2-actions of $\U$ are a main subject of \cite{ManionRouquier} (see also \cite{DM,DLM}).
\begin{definition}
Let $\U$ denote the strict differential monoidal category with objects generated under $\otimes$ by a single object $e$ and with morphisms generated under $\otimes$ and composition by an endomorphism $\tau$ of $e \otimes e$, subject to the relations $\tau^2 = 0$ and
\[
(\id_e \otimes \tau) \circ (\tau \otimes \id_e) \circ (\id_e \otimes \tau) = (\tau \otimes \id_e) \otimes (\id_e \otimes \tau) \otimes (\tau \otimes \id_e),
\]
and with differential determined by $d(\tau) = \id_{e \otimes e}$.
\end{definition}

\begin{remark}
A grading on $\U$ is defined in \cite{KhovOneHalf}, making it into a dg category. Here we will not need to work with this grading; indeed, in the 2-actions of $\U$ we consider below, $\tau$ will act as zero.
\end{remark}

The endomorphism algebra in $\U$ of $e^{\otimes m}$ is the nilCoxeter dg algebra denoted by $\mathfrak{N}_m$ in \cite{DM}.

\subsubsection{2-representations}

We will be especially concerned with 2-representations of $\U$ on associative algebras in the setting of $DA$ bimodules; we give a concrete definition of this notion below.

\begin{definition}
Let $\A$ be an associative algebra (we make the same assumptions on $\A$ as in Section~\ref{sec:DABimod}). A ($DA$ bimodule) 2-representation of $\U$ on $\A$ is the data of a $DA$ bimodule $\E$ over $\A$ and a (typically non-closed) $DA$ bimodule morphism $\tau$ from $\E \boxtimes \E$ to itself satisfying $\tau^2 = 0$, 
\[
(\id_{\E} \boxtimes \tau) \circ (\tau \boxtimes \id_{\E}) \circ (\id_{\E} \boxtimes \tau) = (\tau \boxtimes \id_{\E}) \circ (\id_{\E} \boxtimes \tau) \circ (\tau \boxtimes \id_{\E}),
\]
and $d(\tau) = 1$. We also assume that $\E$ is left bounded in the sense of \cite[Definition 2.2.46]{LOTBimodules}.
\end{definition}

We will write the above data as $(\A,\E,\tau)$. 

\begin{remark}
The definitions of $DA$ bimodule morphisms, their tensor products, and their differentials can be found in \cite[Section 2.2.4 and Section 2.3.2]{LOTBimodules}, but we will refrain from spelling out these definitions here because in the examples we will consider, $\E \boxtimes \E$ will be the zero $DA$ bimodule and $\tau$ will be the zero morphism.
\end{remark}

\subsubsection{1-morphisms of 2-representations}

We will also work with a $DA$ bimodule version of 1-morphisms between 2-representations of $\U$.

\begin{definition}
Let $(\A,\E,\tau)$ and $(\A',\E',\tau')$ be ($DA$ bimodule) 2-representations of $\U$ on associative algebras $\A$ and $\A'$. A ($DA$ bimodule) 1-morphism of 2-representations from $(\A,\E,\tau)$ to $(\A',\E',\tau')$ consists of a left bounded $DA$ bimodule $X$ over $(\A',\A)$ together with a homotopy equivalence
\[
\alpha \colon X \boxtimes \E \to \E' \boxtimes X,
\]
satisfying
\[
(\tau' \boxtimes \id_X) \circ (\id_{\E'} \boxtimes \alpha) \circ (\alpha \boxtimes \id_{\E}) = (\id_{\E'} \boxtimes \alpha) \circ (\alpha \boxtimes \id_{\E}) \circ (\id_X \boxtimes \tau)
\]
as morphisms from $X \boxtimes \E \boxtimes \E$ to $\E' \boxtimes \E' \boxtimes X$.
\end{definition}

\begin{remark}
We will not elaborate on the definition of homotopy equivalence of $DA$ bimodules here (it can be found in \cite[Section 2.2.4]{LOTBimodules}); in this paper the homotopy equivalences $\alpha$ will be isomorphisms given by bijections between primary matrix entries such that the corresponding secondary matrices agree.
\end{remark}

\subsection{Actions on bordered HFK algebras}

In \cite{ManionRouquier}, 2-representations of $\U$ are defined on the algebras $\A(\Zc)$ appearing in bordered sutured Heegaard Floer homology. Here $\Zc$ denotes an arc diagram, i.e. a finite collection of oriented intervals and circles equipped with a 2-to-1 matching of finitely many points in the interiors of the intervals and circles, and there is a 2-representation of $\U$ on $\A(\Zc)$ for each interval in $\Zc$.

The algebra $\B(2)$ was shown in \cite{MMW2,LP} to be quasi-isomorphic to $\A(\Zc)$ where $\Zc$ is the arc diagram shown in Figure~\ref{fig:ArcDiag}. Since $\Zc$ has two intervals, we should expect two 2-actions of $\U$ on $\B(2)$; we define these 2-actions below. See \cite{LaudaManion} for a related 2-representation of $\U$ on an $n$-strand Ozsv{\'a}th--Szab{\'o} algebra from \cite{OSzNew}. In more detail, we will define $DA$ bimodules $\E_1$ and $\E_2$ over $\B(2)$; these bimodules will satisfy $\E_i \boxtimes \E_i = 0$, so that $(\A,\E_i,0)$ is a 2-representation of $\U$.

\begin{remark}
The arc diagram shown in Figure~\ref{fig:ArcDiag} can also be seen on the front and back edges of the Heegaard diagrams in Figure~\ref{fig:HeegaardDiagPositive} and Figure~\ref{fig:HeegaardDiagNegative}, with the red arcs in Figure~\ref{fig:ArcDiag} determined by the matching pattern of the red arcs in the Heegaard diagrams.
\end{remark}

\begin{figure}
    \centering
    \includegraphics[scale=0.5]{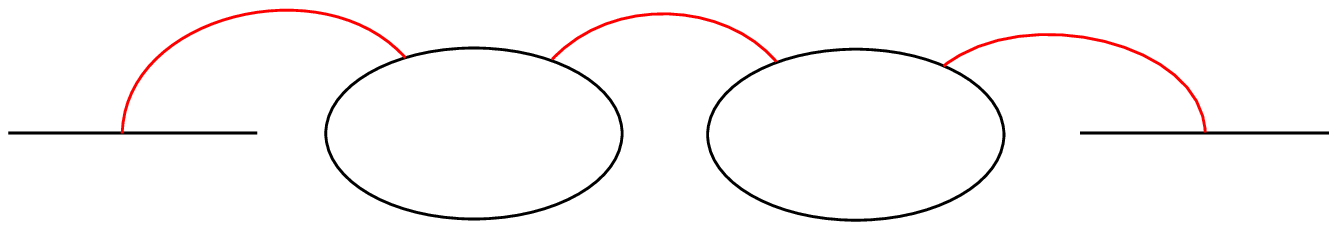}
    \caption{The arc diagram $\Zc$ such that $\A(\Zc)$ is quasi-isomorphic to $\B(2)$; the 2-to-1 matching is indicated by the red arcs, and by symmetry one may take any orientation on the circles and intervals.}
    \label{fig:ArcDiag}
\end{figure}

\begin{definition}
The primary matrix for $\E_1$ has block form with the following blocks (we write e.g. $X_1$ for the singleton set $\{X_1\}$):
\[
    \kbordermatrix{
    & A & B & C \\
    \varnothing & X_1 & \varnothing & \varnothing
    },
    \kbordermatrix{
    & AB & AC & BC \\
    A & \varnothing & \varnothing & \varnothing \\
    B & X_2 & \varnothing & \varnothing \\
    C & \varnothing & X_3 & \varnothing
    },
    \kbordermatrix{
    & ABC \\
    AB & \varnothing \\
    AC & \varnothing \\
    BC & X_4
    }.
\]
The secondary matrix for $\E_1$ has a corresponding block form with blocks
\[
   \kbordermatrix{
   & X_1\\
   X_1 & 0
   },
\]
\[
   \kbordermatrix{ & X_2 & X_3\\
    X_2&  U_1^{k+1} \otimes U_1^{k+1} + U_2^{k+1} \otimes U_2^{k+1} & L_2U_2^k \otimes L_2 U_2^k\\
X_3 &  R_2U_2^k \otimes R_2U_2^k & U_2^{k+1}\otimes U_2^{k+1}},
\]
\[
\kbordermatrix{& X_4\\
X_4 & U_1^kU_2^l \otimes U_1^kU_2^l};
\]
in the final block we disallow $(k,l) = (0,0)$ to match Convention~\ref{conv:UnitsOmitted}.
\end{definition}

\begin{definition}
The primary matrix for $\E_2$ has block form with the following blocks (again we write e.g. $Y_1$ for the singleton set $\{Y_1\}$):
\[
    \kbordermatrix{
    & A & B & C \\
    \varnothing & \varnothing & \varnothing & Y_1
    },
    \kbordermatrix{
    & AB & AC & BC \\
    A & \varnothing & Y_2 & \varnothing \\
    B & \varnothing & \varnothing & Y_3\\
    C & \varnothing & \varnothing & \varnothing
    },
    \kbordermatrix{
    & ABC \\
    AB & Y_4 \\
    AC & \varnothing \\
    BC & \varnothing
    }.
\]
The secondary matrix for $\E_2$ has a corresponding block form with blocks
\[
    \kbordermatrix{& Y_1\\
    Y_1 & 0},
\]
\[
    \kbordermatrix{& Y_2 & Y_3\\
    Y_2 & U_1^{k+1} \otimes U_1^{k+1} & L_1U_1^k \otimes L_1U_1^k\\
    Y_3 & R_1U_1^k \otimes R_1U_1^k & U_1^{k+1}\otimes U_1^{k+1} + U_2^{k+1}\otimes U_2^{k+1}},
\]
\[
    \kbordermatrix{& Y_4\\
   Y_4 & U_1^kU_2^l \otimes U_1U_2^l };
\]
in the final block we disallow $(k,l) = (0,0)$ to match Convention~\ref{conv:UnitsOmitted}.
\end{definition}

By multiplying the primary matrix for $\E_i$ by itself ($i = 1,2$), one can see that $\E_i \boxtimes \E_i$ has a primary matrix with each entry the empty set; in other words, $\E_i \boxtimes \E_i$ is zero as claimed above.

\section{1-morphism structure for \texorpdfstring{$\Pc$}{P}}\label{sec:P1Morphism}

\subsection{Commutativity with \texorpdfstring{$\E_1$}{E1}}

\subsubsection{The bimodule $\E_1 \boxtimes \Pc$}

We give a matrix description for $\E_1 \boxtimes \Pc$ following Section~\ref{sec:MatrixBoxTensor}. To get the primary matrix for $\E_1 \boxtimes \Pc$, we multiply the primary matrices for $\E_1$ and $\Pc$. We can do this block-by-block, so the primary matrix for $\E_1 \boxtimes \Pc$ has block form with blocks given by
\[
    \kbordermatrix{
     & A & B & C \\
    \varnothing & X_1 & \varnothing & \varnothing
    }
    \cdot \kbordermatrix{
     & A & B & C \\
    A & S_A & \varnothing & \varnothing \\
    B & W & N & E \\
    C & \varnothing & \varnothing & S_C
    }
    = \kbordermatrix{
     & A & B & C \\
    \varnothing & X_1 S_C & \varnothing & \varnothing
    },
\]
\[
    \kbordermatrix{
     & AB & AC & BC \\
    A & \varnothing & \varnothing & \varnothing \\
    B & X_2 & \varnothing & \varnothing \\
    C & \varnothing & X_3 & \varnothing
    }
    \cdot \kbordermatrix{
     & AB & AC & BC \\
    AB & N_{AB} & E & \varnothing \\
    AC & \varnothing & S & \varnothing \\
    BC & \varnothing & W & N_{BC}
    }
    = \kbordermatrix{
     & AB & AC & BC \\
    A & \varnothing & \varnothing & \varnothing \\
    B & X_2 N_{AB} & X_2 E & \varnothing \\
    C & \varnothing & X_3 X & \varnothing
    },
\]
\[
    \kbordermatrix{
     & ABC \\
    AB & \varnothing \\
    AC & \varnothing \\
    BC & X_4
    }
    \cdot \kbordermatrix{
     & ABC \\
     ABC & N
    }
    = \kbordermatrix{
     & ABC \\
     AB & \varnothing \\
     AC & \varnothing \\
     BC & X_4 N
    }.
\]
In these matrices, we indicate idempotents only when necessary to distinguish primary matrix entries in the same block (so, for example, in the block with rows and columns $A,B,C$, we distinguish between two types of $S$ generators, but the only $N$ generator in this block is ${_{B}}N_B$ so we omit the idempotents and just write $N$).

The secondary matrix for $\E_1 \boxtimes \Pc$ also has block form with blocks given by
\[
    \kbordermatrix{
    & X_1 S_C \\
    X_1 S_C & 0
    },
\]
\[
    \kbordermatrix{
    & X_2 N_{AB} & X_2 E & X_3 S  \\
   X_2 N_{AB} & U_2^{k+1} \otimes U_1^{k+1} + U_1^{k+1} \otimes U_2^{k+1} &  U_1^k \otimes L_2U_2^k & L_2 U_2^k\otimes (L_2, U_1^{k+1})\\
    X_2 E & U_1^{k+1} \otimes R_2U_2^k & U_1^{k+1} \otimes U_2^{k+1}& 0\\
    X_3 S & 0 & R_2 & 0
    },
\]
\[
    \kbordermatrix{
     & X_4N \\
     X_4N & U_1^l U_2^k \otimes U_1^k U_2^l
    };
\]
in the final block we disallow $(k,l) = (0,0)$. An explanation for the terms in the secondary matrix is given in Figure~\ref{fig:E1PTrees}, which uses the operation graph depictions of Figure~\ref{fig:BoxTensorTree}.

\begin{figure}
    \centering
    \includegraphics[scale=0.8]{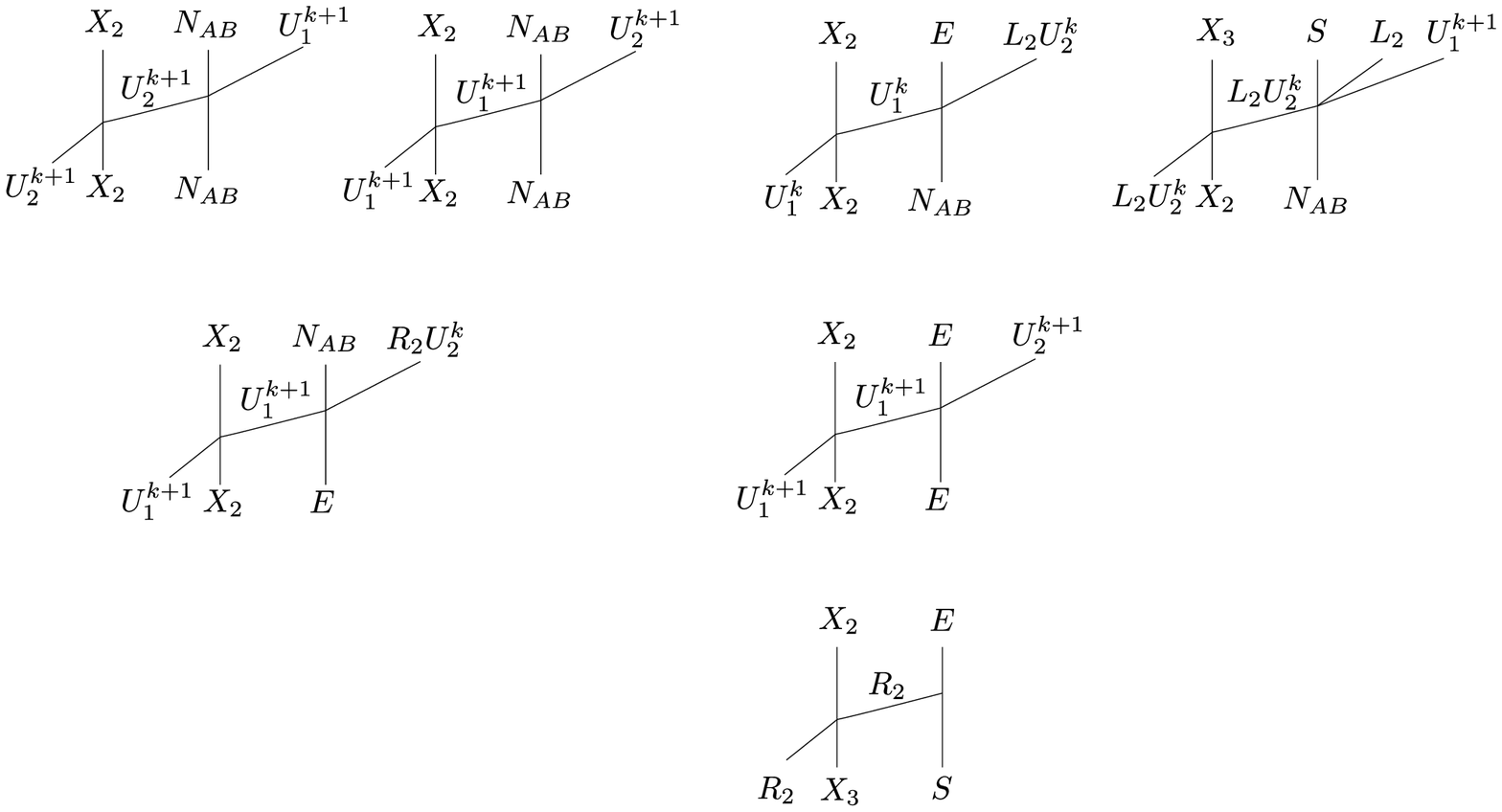}
    \caption{Operation graphs for the terms in the secondary matrix of $\E_1 \boxtimes \Pc$.}
    \label{fig:E1PTrees}
\end{figure}

\subsubsection{The bimodule $\Pc \boxtimes \E_1$}

Similarly, we give a matrix description for $\Pc \boxtimes \E_1$. The primary matrix has block form with blocks
\[
\kbordermatrix{
 & \varnothing \\
\varnothing & S
}
\cdot \kbordermatrix{
 & A & B & C \\
\varnothing & X_1 & \varnothing & \varnothing
}
= \kbordermatrix{
 & A & B & C \\
\varnothing & S X_1 & \varnothing & \varnothing
},
\]
\[
\kbordermatrix{
 & A & B & C \\
A & S_A & \varnothing & \varnothing \\
B & W & N & E \\
C & \varnothing & \varnothing & S_C
}
\cdot \kbordermatrix{
 & AB & AC & BC \\
A & \varnothing & \varnothing & \varnothing \\
B & X_2 & \varnothing & \varnothing \\
C & \varnothing & X_3 & \varnothing
}
= \kbordermatrix{
 & AB & AC & BC \\
A & \varnothing & \varnothing & \varnothing \\
B & N X_2 & E X_3 & \varnothing \\
C & \varnothing & S_C X_3 & \varnothing
},
\]
\[
\kbordermatrix{
 & AB & AC & BC \\
AB & N_{AB} & E & \varnothing \\
AC & \varnothing & S & \varnothing \\
BC & \varnothing & W & N_{BC}
}
\cdot \kbordermatrix{
 & ABC \\
AB & \varnothing \\
AC & \varnothing \\
BC & X_4
}
= \kbordermatrix{
 & ABC \\
AB & \varnothing \\
AC & \varnothing \\
BC & N_{BC} X_4
}.
\]

The secondary matrix for $\Pc \boxtimes \E_1$ also has block form with blocks
\[
    \kbordermatrix{
    & SX_1 \\
    SX_1 & 0
    },
\]
\[
    \kbordermatrix{
    & N X_2 & E X_3 & S_C X_3  \\
    N X_2 & U_2^{k+1} \otimes U_1^{k+1} + U_1^{k+1} \otimes U_2^{k+1} &  U_1^k \otimes L_2U_2^k & L_2U_2^k \otimes (L_2, U_1^{k+1})\\
    E X_3 & U_1^{k+1} \otimes R_2U_2^k & U_1^{k+1} \otimes U_2^{k+1}& 0\\
    S_C X_3 & 0 & R_2 & 0
    },
\]
\[
    \kbordermatrix{
     & N_{BC} X_4 \\
     N_{BC} X_4 & U_1^l U_2^k \otimes U_1^k U_2^l
    };
\]
in the final block we disallow $(k,l) = (0,0)$. An explanation for the terms in the secondary matrix is given in Figure~\ref{fig:PE1Trees}.

\begin{figure}
    \centering
    \includegraphics[scale=0.8]{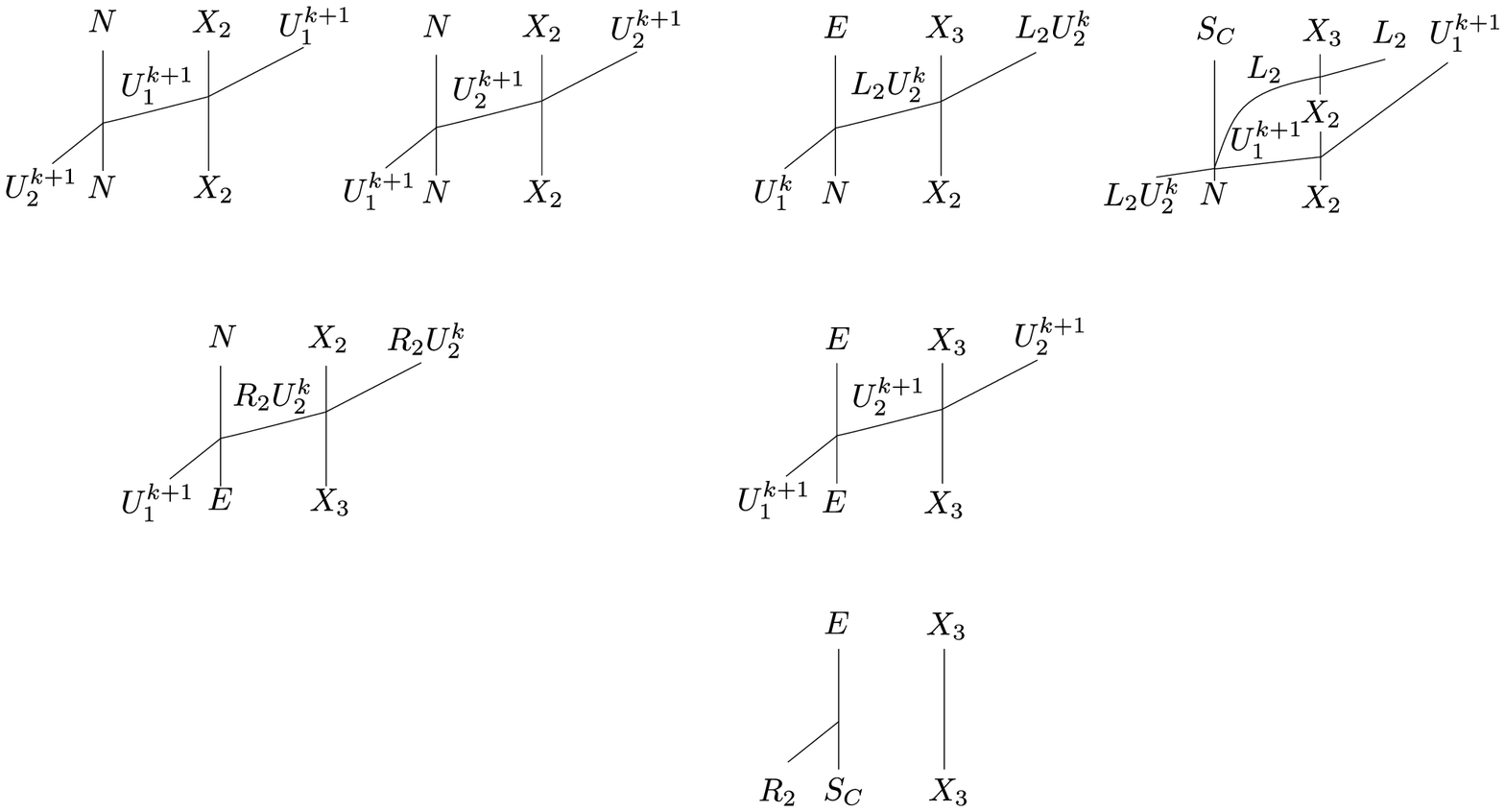}
    \caption{Operation graphs for the terms in the secondary matrix of $\Pc \boxtimes \E_1$.}
    \label{fig:PE1Trees}
\end{figure}

\begin{corollary}
The $DA$ bimodules $\E_1 \boxtimes \Pc$ and $\Pc \boxtimes \E_1$ are isomorphic to each other.
\end{corollary}

\begin{proof}
The primary and secondary matrices for $\E_1 \boxtimes \Pc$ and $\Pc \boxtimes \E_1$ agree up to a relabeling of primary matrix entries.
\end{proof}

\subsection{Commutativity with \texorpdfstring{$\E_2$}{E2}}

\subsubsection{The bimodule $\E_2 \boxtimes \Pc$}

Next we give a matrix description of $\E_2 \boxtimes \Pc$. The primary matrix has block form with blocks
\[
\kbordermatrix{
 & A & B & C \\
\varnothing & \varnothing & \varnothing & Y_1
}
\cdot \kbordermatrix{
 & A & B & C \\
A & S_A & \varnothing & \varnothing \\
B & W & N & E \\
C & \varnothing & \varnothing & S_C
}
= \kbordermatrix{
 & A & B & C \\
\varnothing & \varnothing & \varnothing & Y_1 S_C
},
\]
\[
\kbordermatrix{
 & AB & AC & BC \\
A & \varnothing & Y_2 & \varnothing \\
B & \varnothing & \varnothing & Y_3 \\
C & \varnothing & \varnothing & \varnothing
}
\cdot \kbordermatrix{
 & AB & AC & BC \\
AB & N_{AB} & E & \varnothing \\
AC & \varnothing & S & \varnothing \\
BC & \varnothing & W & N_{BC}
}
= \kbordermatrix{
 & AB & AC & BC \\
A & \varnothing & Y_2 S & \varnothing \\
B & \varnothing & Y_3 W & Y_3 N_{BC} \\
C & \varnothing & \varnothing & \varnothing
},
\]
\[
\kbordermatrix{
 & ABC \\
AB & Y_4 \\
AC & \varnothing \\
BC & \varnothing
}
\cdot \kbordermatrix{
 & ABC \\
ABC & N
}
= \kbordermatrix{
 & ABC \\
AB & Y_4 N \\
AC & \varnothing \\
BC & \varnothing
}.
\]

The secondary matrix for $\E_2 \boxtimes \Pc$ also has block form with blocks
\[
    \kbordermatrix{ & Y_1 S_C\\
    Y_1 S_C & 0},
\]
\[
    \kbordermatrix{ & Y_2 S & Y_3 W &  Y_3 N_{BC}\\
    Y_2 S & 0 & L_1 & 0\\
    Y_3 W & 0 & U_2^{k+1} \otimes U_1^{k+1} & U_2^{k+1} \otimes L_1U_1^k \\
    Y_3 N_{BC} & R_1U_1^k\otimes (R_1, U_2^{k+1}) & U_2^{k} \otimes R_1 U_1^{k} & U_1^{k+1} \otimes U_2^{k+1} + U_2^{k+1} \otimes U_1^{k+1}},
\]
\[
    \kbordermatrix{ & Y_4 N\\
    Y_4 N & U_1^kU_2^l \otimes U_1^l U_2^k};
\]
in the final block we disallow $(k,l) = (0,0)$. One can draw operation graphs for the secondary matrix entries as we did above in Figures~\ref{fig:E1PTrees} and \ref{fig:PE1Trees}, but we will omit the graphs here.

\subsubsection{The bimodule $\Pc \boxtimes \E_2$}

The primary matrix for $\Pc \boxtimes \E_2$ has block form with blocks
\[
\kbordermatrix{
 & \varnothing \\
\varnothing & S
}
\cdot \kbordermatrix{
 & A & B & C \\
\varnothing & \varnothing & \varnothing & Y_1
}
= \kbordermatrix{
 & A & B & C \\
\varnothing & \varnothing & \varnothing & S Y_1
},
\]
\[
\kbordermatrix{
 & A & B & C \\
A & S_A & \varnothing & \varnothing \\
B & W & N & E \\
C & \varnothing & \varnothing & S_C
}
\cdot \kbordermatrix{
 & AB & AC & BC \\
A & \varnothing & Y_2 & \varnothing \\
B & \varnothing & \varnothing & Y_3 \\
C & \varnothing & \varnothing & \varnothing
}
= \kbordermatrix{
 & AB & AC & BC \\
A & \varnothing & S_A Y_2 & \varnothing \\
B & \varnothing & W Y_2 & N Y_3 \\
C & \varnothing & \varnothing & \varnothing
},
\]
\[
\kbordermatrix{
 & AB & AC & BC \\
AB & N_{AB} & E & \varnothing \\
AC & \varnothing & S & \varnothing \\
BC & \varnothing & W & N_{BC}
}
\cdot \kbordermatrix{
 & ABC \\
AB & Y_4 \\
AC & \varnothing \\
BC & \varnothing
}
= \kbordermatrix{
 & ABC \\
AB & N_{AB} Y_4 \\
AC & \varnothing \\
BC & \varnothing
}.
\]

The secondary matrix for $\Pc \boxtimes \E_2$ also has block form with blocks
\[
    \kbordermatrix{ & SY_1\\
    SY_1 & 0},
\]
\[
      \kbordermatrix{ & S_A Y_2 & W Y_2 & N Y_3 \\
    S_A Y_2 & 0 & L_1 & 0\\
    W Y_2 & 0 & U_2^{k+1} \otimes U_1^{k+1} & U_2^{k+1} \otimes L_1U_1^k \\
    N Y_3 & R_1U_1^k\otimes (R_1, U_2^{k+1}) & U_2^{k} \otimes R_1 U_1^{k} & U_1^{k+1} \otimes U_2^{k+1} + U_2^{k+1} \otimes U_1^{k+1}},
\]
\[
    \kbordermatrix{ & N_{AB} Y_4 \\
    N_{AB} Y_4 & U_1^k U_2^l \otimes U_1^l U_2^k};
\]
in the final block we disallow $(k,l) = (0,0)$. As with $\E_2 \boxtimes \Pc$, we will omit drawing the operation graphs.

\begin{corollary}
The $DA$ bimodules $\E_2 \boxtimes \Pc$ and $\Pc \boxtimes \E_2$ are isomorphic to each other.
\end{corollary}

\begin{proof}
The primary and secondary matrices for $\E_2 \boxtimes \Pc$ and $\Pc \boxtimes \E_2$ agree up to a relabeling of primary matrix entries.
\end{proof}

\section{1-morphism structure for \texorpdfstring{$\Nc$}{N}}\label{sec:N1Morphism}

Here we summarize, with fewer details, the computations for $\Nc$ that are analogous to those for $\Pc$ in Section~\ref{sec:P1Morphism}.

\subsection{Commutativity with \texorpdfstring{$\E_1$}{E1}}

\subsubsection{The bimodule $\E_1 \boxtimes \Nc$}

The primary matrix for $\E_1 \boxtimes \Nc$ has block form with the same blocks as for $\E_1 \boxtimes \Pc$, namely
\[
\kbordermatrix{
& A & B & C \\
\varnothing & \varnothing & \varnothing & X_1 S_C
}, \quad
\kbordermatrix{
& AB & AC & BC \\
A & \varnothing & \varnothing & \varnothing \\
B & X_2 N_{AB} & X_2 E & \varnothing \\
C & \varnothing & X_3 X & \varnothing
}, \quad
\kbordermatrix{
& ABC \\
AB & \varnothing \\
AC & \varnothing \\
BC & X_4 N
}.
\]
The secondary matrix for $\E_1 \boxtimes \Nc$ has block form with blocks given by
\[
    \kbordermatrix{
    & X_1 S_C \\
    X_1 S_C & 0
    },
\]
\[
    \kbordermatrix{
    & X_2 N_{AB} & X_2 E & X_3 S  \\
   X_2 N_{AB} & U_2^{k+1} \otimes U_1^{k+1} + U_1^{k+1} \otimes U_2^{k+1} &  U_1^{k+1} \otimes L_2U_2^k & 0 \\
    X_2 E & U_1^k \otimes R_2U_2^k & U_1^{k+1} \otimes U_2^{k+1}& L_2 \\
    X_3 S & R_2 U_2^k \otimes (U_1^{k+1}, R_2) & 0 & 0
    },
\]
\[
    \kbordermatrix{
     & X_4N \\
     X_4N & U_1^l U_2^k \otimes U_1^k U_2^l
    };
\]
in the final block we disallow $(k,l) = (0,0)$.

\subsubsection{The bimodule $\Nc \boxtimes \E_1$}

The primary matrix for $\Nc \boxtimes \E_1$ has block form with the same blocks as for $\Pc \boxtimes \E_1$, namely
\[
\kbordermatrix{
 & A & B & C \\
\varnothing & S X_1 & \varnothing & \varnothing
}, \quad
\kbordermatrix{
 & AB & AC & BC \\
A & \varnothing & \varnothing & \varnothing \\
B & N X_2 & E X_3 & \varnothing \\
C & \varnothing & S_C X_3 & \varnothing
}, \quad
\kbordermatrix{
 & ABC \\
AB & \varnothing \\
AC & \varnothing \\
BC & N_{BC} X_4
}.
\]
The secondary matrix for $\Nc \boxtimes \E_1$ has block form with blocks given by
\[
    \kbordermatrix{
    & SX_1 \\
    SX_1 & 0
    },
\]
\[
    \kbordermatrix{
    & N X_2 & E X_3 & S_C X_3  \\
    N X_2 & U_2^{k+1} \otimes U_1^{k+1} + U_1^{k+1} \otimes U_2^{k+1} &  U_1^{k+1} \otimes L_2U_2^k & 0 \\
    E X_3 & U_1^k \otimes R_2U_2^k & U_1^{k+1} \otimes U_2^{k+1} & L_2 \\
    S_C X_3 & R_2 U_2^k \otimes (U_1^{k+1}, R_2) & 0 & 0
    },
\]
\[
    \kbordermatrix{
     & N_{BC} X_4 \\
     N_{BC} X_4 & U_1^l U_2^k \otimes U_1^k U_2^l
    };
\]
in the final block we disallow $(k,l) = (0,0)$.

\begin{corollary}
The $DA$ bimodules $\E_1 \boxtimes \Nc$ and $\Nc \boxtimes \E_1$ are isomorphic to each other.
\end{corollary}

\subsection{Commutativity with \texorpdfstring{$\E_2$}{E2}}

\subsubsection{The bimodule $\E_2 \boxtimes \Nc$}

The primary matrix for $\E_2 \boxtimes \Nc$ has block form with the same blocks as for $\E_2 \boxtimes \Pc$, namely
\[
\kbordermatrix{
 & A & B & C \\
\varnothing & \varnothing & \varnothing & Y_1 S_C
}, \quad
\kbordermatrix{
 & AB & AC & BC \\
A & \varnothing & Y_2 S & \varnothing \\
B & \varnothing & Y_3 W & Y_3 N_{BC} \\
C & \varnothing & \varnothing & \varnothing
}, \quad
\kbordermatrix{
 & ABC \\
AB & Y_4 N \\
AC & \varnothing \\
BC & \varnothing
}.
\]
The secondary matrix for $\E_2 \boxtimes \Nc$ has block form with blocks given by
\[
    \kbordermatrix{ & Y_1 S_C\\
    Y_1 S_C & 0},
\]
\[
    \kbordermatrix{ & Y_2 S & Y_3 W &  Y_3 N_{BC}\\
    Y_2 S & 0 & 0 & L_1 U_1^k \otimes (U_2^{k+1}, L_1) \\
    Y_3 W & R_1 & U_2^{k+1} \otimes U_1^{k+1} & U_2^k \otimes L_1U_1^k \\
    Y_3 N_{BC} &0 & U_2^{k+1} \otimes R_1 U_1^{k} & U_1^{k+1} \otimes U_2^{k+1} + U_2^{k+1} \otimes U_1^{k+1}},
\]
\[
    \kbordermatrix{ & Y_4 N\\
    Y_4 N & U_1^kU_2^l \otimes U_1^l U_2^k};
\]
in the final block we disallow $(k,l) = (0,0)$.

\subsubsection{The bimodule $\Nc \boxtimes \E_2$}

The primary matrix for $\Nc \boxtimes \E_2$ has block form with the same blocks as for $\Pc \boxtimes \E_2$, namely
\[
\kbordermatrix{
 & A & B & C \\
\varnothing & \varnothing & \varnothing & S Y_1
}, \quad
\kbordermatrix{
 & AB & AC & BC \\
A & \varnothing & S_A Y_2 & \varnothing \\
B & \varnothing & W Y_2 & N Y_3 \\
C & \varnothing & \varnothing & \varnothing
}, \quad
\kbordermatrix{
 & ABC \\
AB & N_{AB} Y_4 \\
AC & \varnothing \\
BC & \varnothing
}.
\]
The secondary matrix for $\Nc \boxtimes \E_2$ has block form with blocks given by
\[
    \kbordermatrix{ & SY_1\\
    SY_1 & 0},
\]
\[
      \kbordermatrix{ & S_A Y_2 & W Y_2 & N Y_3 \\
    S_A Y_2 & 0 & 0 & L_1 U_1^k \otimes (U_2^{k+1}, L_1) \\
    W Y_2 & R_1 & U_2^{k+1} \otimes U_1^{k+1} & U_2^k \otimes L_1U_1^k \\
    N Y_3 & 0 & U_2^{k+1} \otimes R_1 U_1^{k} & U_1^{k+1} \otimes U_2^{k+1} + U_2^{k+1} \otimes U_1^{k+1}},
\]
\[
    \kbordermatrix{ & N_{AB} Y_4 \\
    N_{AB} Y_4 & U_1^k U_2^l \otimes U_1^l U_2^k};
\]
in the final block we disallow $(k,l) = (0,0)$.

\begin{corollary}
The $DA$ bimodules $\E_2 \boxtimes \Nc$ and $\Nc \boxtimes \E_2$ are isomorphic to each other.
\end{corollary}

\bibliographystyle{alpha}
\bibliography{biblio.bib}

\end{document}